\documentclass{article}
\usepackage[top=1in,bottom=1.1in,left=1in,right=1in]{geometry}
\usepackage{amssymb}
\usepackage{amsmath,color}
\usepackage{amsthm}
\usepackage{verbatim}
\usepackage{epsfig}
\usepackage{hyperref}
\usepackage[T1]{fontenc}
\usepackage{enumitem}

\newtheorem{theorem}{Theorem}[section]
\newtheorem{proposition}[theorem]{Proposition}
\newtheorem{corollary}[theorem]{Corollary}

\newtheorem{claim}{Claim}
\newtheorem{question}{Question}

\theoremstyle{definition}
\newtheorem{definition}[theorem]{Definition}
\theoremstyle{remark}
\newtheorem{remark}[theorem]{Remark}
\newtheorem{example}[theorem]{Example}

\newcommand{\abs}[1]{\left\vert#1\right\vert}  
\newcommand{\dist}{\operatorname{dist}}   
\newcommand{\CAT}{\operatorname{CAT}}

\newcommand{\N}{\mathbb{N}}                                            % Symbol of natural numbers
\newcommand{\R}{\mathbb{R}}                                         % Symbol of real numbers
\newcommand{\eps}{\varepsilon} 

\begin{document}

\title{Geodesic rays, the ``Lion-Man'' game, and the fixed point property}

\author{Genaro L\'{o}pez-Acedo$^{a}$, Adriana Nicolae$^{b}$, Bo\.{z}ena Pi\k{a}tek$^{c}$}
\date{}
\maketitle

\begin{center}
{\footnotesize
$^{a}$Department of Mathematical Analysis - IMUS, University of Seville, Sevilla, Spain\\
$^{b}$Department of Mathematics, Babe\c s-Bolyai University, Kog\u alniceanu 1, 400084 Cluj-Napoca, Romania\\
$^{c}$Institute of Mathematics, Silesian University of Technology, 44-100 Gliwice, Poland\\
\ \\
E-mail addresses:  glopez@us.es (G. L\'{o}pez-Acedo), anicolae@math.ubbcluj.ro (A. Nicolae), Bozena.Piatek@polsl.pl (B. Pi\k{a}tek).
}
\end{center}

\begin{abstract}
This paper focuses on the relation among the existence of different types of curves (such as directional ones, quasi-geodesic or geodesic rays), the (approximate) fixed point property for nonexpansive mappings, and a discrete lion and man game. Our main result holds in the setting of $\CAT(0)$ spaces that are additionally Gromov hyperbolic.
\end{abstract}

\section{Introduction}
Given a metric space $X$, the existence of either a geodesic ray in $X$ or, more generally, of a curve that approximates a geodesic ray (e.g., a directional curve or a quasi-geodesic ray, see Section \ref{section-rays} for precise definitions), is connected to some intrinsic topological and geometric properties of the space. Such properties include the presence of upper or lower curvature bounds in the sense of Alexandrov \cite{Bri99,EspPia13,LopPia15}, the Gromov hyperbolicity condition \cite{Bri99}, the betwenness property \cite{LopNicPia18}, or, for normed spaces, the reflexivity \cite{Sha90}. On the other hand, the existence of different types of curves can be used in the analysis of several problems among which we mention the characterization of the fixed point property either for continuous \cite{Klee55,Ki04,LopPia15} or nonexpansive mappings \cite{Sha90,EspPia13,Pia15}, and some pursuit-evasion games \cite{AleBisGhr10,LopNicPia18}. 

Geometric properties of a set stand behind the existence of fixed points for continuous or nonexpansive mappings defined on the set in question. This fact has prompted a very fruitful research direction because, among other reasons, it leads to challenging questions regarding the geometry of Banach spaces (see, e.g., the monographs \cite{ADL97, GoeKir90}). Klee \cite{Klee55} was a pioneer in the use of topological rays as a tool to characterize compactness of convex subsets of a locally convex metrizable linear topological space by means of the fixed point property for continuous mappings (see also \cite{Db97}). Counterparts of this result for geodesic metric spaces were proved in \cite{LopPia15,LopPia16}. The notion of directional curves was introduced in \cite{Sha90} to characterize the approximate fixed point property for nonexpansive mappings in convex subsets of a class of metric spaces which includes, in particular, Banach spaces or complete Busemann convex geodesic spaces.  Further results in this direction were proved in \cite{Ki04,Pia15}. Namely, the absence of geodesic rays from a closed and convex set is equivalent to its fixed point property for continuous mappings in complete $\R$-trees, while for the more general setting of complete $\CAT(\kappa)$ spaces with $\kappa < 0$, it is equivalent to its fixed point property for nonexpansive mappings.

Persuit-evasion problems are not only interesting because their analysis requires the use of tools from different areas of mathematics, but also because of their application in other disciplines such as robotics or the modeling of animal behavior. In a recent survey, Chung et al. \cite{ChHI11} classify persuit-evasion games based on the environment where the game is played, the information available to the players, the restrictions imposed on the players' motion, and the way of defining capture.   

The game that we consider in this work was proposed in \cite{AleBisGhr06,AleBisGhr10} (see also \cite{Bac12}) and is a discrete variant of the classical lion and man game which goes back to Rado (see \cite{Li86}). This discrete equal-speed game with an $\eps$-capture criterion, whose precise rules are described at the beginning of Section \ref{lion-game}, will be called in the sequel the Lion-Man game. A first step in the study of the role that geodesic rays play in the analysis of the Lion-Man game, and consequently its connection with the fixed point property, was given in \cite{LopNicPia18}, where it was shown that in complete, locally compact, uniquely geodesic spaces, assuming a strongly convex domain, the success of the lion, the fixed point property for continuous mappings, and the compactness of the domain are all equivalent (see also \cite{Yuf18}). Further results also from the quantitative point of view have been obtained very recently in \cite{KohLopNic18}. In this work we analyze the relation among the existence of different types of curves and deduce, in light of this analysis, connections among the solution of the Lion-Man game, the (approximate) fixed point property for nonexpansive mappings, and the existence of geodesic rays in the domain where the game is played.

The organization of the paper is as follows. After recalling in Section \ref{Preliminaries} some basic notions on geodesic metric spaces, we study in Section \ref{section-rays} different types of curves which are related to the (approximate) fixed point property for nonexpansive mappings and to the Lion-Man game. More precisely, we focus on directional curves, local quasi-geodesic, quasi-geodesic and geodesic rays. Of particular importance for our main result is Corollary \ref{cor-local-quasi-ray} which states, in broad lines, that in the setting of complete Busemann convex spaces that are additionally Gromov hyperbolic, the existence of a local quasi-geodesic ray implies the existence of a geodesic ray. Section \ref{lion-game} contains the main result of this work, Theorem \ref{main}, which shows that in the setting of complete $\CAT(0)$ spaces that are additionally Gromov hyperbolic, for a closed convex domain $A$, the following are equivalent: $A$ does not contain geodesic rays, $A$ has the fixed point property for nonexpansive mappings, and the lion always wins the Lion-Man game played in $A$. We complete the section with some comments and a result about the finite termination of the game in the setting of $\R$-trees.

\section{Preliminaries}\label{Preliminaries} 

We include here a brief despription of some of the notions and properties of geodesic metric spaces that we will use in the following sections (for a detailed discussion, see, e.g., \cite{Bri99,Pap05}). 

Let $(X,d)$ be a metric space. For $x \in X$ and a nonempty subset $A$ of $X$, we denote the {\it distance} of $x$ to $A$ by $\text{dist}(x,A) = \inf\{d(x,a) : a \in A\}$ and the {\it metric projection} of $x$ onto $A$ by $P_A(x)=\{ y \in A : d(x,y)=\mbox{dist}(x,A)\}$. 

Let $x,y \in X$. A {\it geodesic} joining $x$ and $y$ is a mapping $\gamma:[a,b]\subseteq {\mathbb R}\to X$ such that $\gamma(a)=x$, $\gamma(b)=y$ and $d(\gamma(s),\gamma(t))=\abs{s-t}$ for all $s,t\in[a,b]$. In this case we also say that $\gamma$ {\it starts from} $x$. The image $\gamma([a,b])$ of a geodesic $\gamma$ joining $x$ and $y$ is called a {\it geodesic segment} joining $x$ and $y$. A point $z\in X$ belongs to a geodesic segment joining $x$ and $y$ if and only if there exists $t\in [0,1]$ such that $d(z,x)= td(x,y)$ and $d(z,y)=(1-t)d(x,y)$. In this case $z = \gamma((1-t)a+tb)$, where $\gamma : [a, b] \to X$ is a geodesic joining $x$ and $y$ (that starts from $x$) and whose image is the geodesic segment in question. We denote a geodesic segment joining $x$ and $y$ by $[x,y]$. Note however that, in general, geodesic segments between two given points might not be unique. If every two points in $X$ are joined by a (unique) geodesic segment, we say that $X$ is a {\it (uniquely) geodesic space}. A subset $A$ of a geodesic space is {\it convex} if given two points in $A$, every geodesic segment joining them is contained in $A$.
 
 If in the definition of a geodesic, instead of the interval $[a,b]$, one considers $[0,\infty)$, then the image of $\gamma$ is called a {\it geodesic ray} with the remark that sometimes we also refer to the mapping $\gamma$ itself as a geodesic ray. We say that a subset of a geodesic space is {\it geodesically bounded} if it does not contain any geodesic ray. 

In normed spaces, algebraic segments are geodesic segments and half-lines are geodesic rays. Thus, every normed space is a geodesic space. Moreover, a normed space is uniquely geodesic if and only if it is strictly convex. In this case, the geodesic segments coincide with the algebraic segments, while the geodesic rays are precisely the half-lines.

A convex subset of a normed space $E$ is called {\it linearly bounded} if it has a bounded intersection with all the lines in $E$ (see, e.g., \cite{Rei83}). Note that in strictly convex normed spaces, the notions of linear boundedness and geodesic boundedness agree.

Suppose next that $(X,d)$ is a geodesic space. We say that $X$ is {\it Busemann convex} if given any two geodesics $\gamma:[a,b] \to X$ and $\sigma:[c,d] \to X$,
\[d(\gamma((1-t)a+tb),\sigma((1-t)c+td)) \le (1-t)d(\gamma(a),\sigma(c)) + td(\gamma(b),\sigma(d)) \quad \text{for any } t \in [0,1].\]
Every Busemann convex space is uniquely geodesic.  In addition, a normed space is Busemann convex if and only if it is strictly convex.

For $\kappa \in \mathbb{R}$, let $M^2_{\kappa}$ be the complete, simply connected $2$-dimensional Riemannian manifold of constant sectional curvature $\kappa$. In the sequel we assume that $\kappa \le 0$.

A {\it geodesic triangle} $\Delta = \Delta(x_1,x_2,x_3)$ in $X$ consists of three points $x_1, x_2,x_3 \in X$ (its {\it vertices}) and three geodesic segments (its {\it sides}) joining each pair of points. A {\it comparison triangle} for $\Delta$ is a triangle $\overline{\Delta} = \Delta(\overline{x}_1, \overline{x}_2, \overline{x}_3)$ in $M^2_{\kappa}$ with $d(x_i,x_j) = d_{M^2_{\kappa}}(\overline{x}_i,\overline{x}_j)$ for $i,j \in \{1,2,3\}$. For $\kappa$ fixed, comparison triangles of geodesic triangles always exist and are unique up to isometry. 

Let $\gamma:[a,b]\to X$ and $\sigma:[c,d] \to X$ be two nonconstant geodesics that start from the same point $x=\gamma(a)=\sigma(c)$. For $t \in (a,b]$, $s \in (c,d]$, and a geodesic triangle $\Delta(x,\gamma(t),\sigma(s))$, consider a comparison triangle $\overline{\Delta}(\overline{x}, \overline{\gamma(t)}, \overline{\sigma(s)})$ in $\mathbb{R}^2=M^2_0$ and denote its interior angle at $\overline{x}$ by $\overline{\angle}_{x}\left(\gamma(t),\sigma(s)\right)$. The {\it Alexandrov angle} $\angle(\gamma,\sigma)$ between the geodesics $\gamma$ and $\sigma$ is defined as 
\[\angle(\gamma,\sigma) = \limsup_{t,s \to 0}\overline{\angle}_{x}\left(\gamma(t),\sigma(s)\right).\] 
For $x,y,z \in X$ with $x \ne y$ and $x \ne z$, if both the points $x$ and $y$, and $x$ and $z$, are joined by a unique geodesic segment, then we also denote the corresponding Alexandrov angle by $\angle_x(y,z)$.

If $\gamma_1 , \gamma_2 ,\gamma_3$ are three geodesics that start from the same point, then  
\[\angle (\gamma_1,\gamma_2 ) \leq \angle (\gamma_1,\gamma_3 )+\angle (\gamma_3,\gamma_2 ).\]
In particular, let $\gamma:[a,b]\to X$ be a nonconstant geodesic and $c \in (a,b)$. Define $\gamma_1:[a,c]\to X$ by $\gamma_1(t)=\gamma(a+c-t)$ and $
\gamma_2:[c,b]\to X$ by $\gamma_2(t)=\gamma(t)$. If $\gamma_3$ is a nonconstant geodesic that starts from $\gamma(c)$, then $\angle(\gamma_1,\gamma_3) + \angle(\gamma_3,\gamma_2) \ge \pi$.

A geodesic triangle $\Delta$ in $X$ satisfies the {\it $\CAT(\kappa)$ inequality} if for every comparison triangle $\overline{\Delta}$ in $M^2_{\kappa}$ of $\Delta$ and for every $x,y \in \Delta$ we have
\[d(x,y) \le d_{M^2_{\kappa}}(\overline{x},\overline{y}),\]
where  $\overline{x},\overline{y} \in \overline{\Delta}$ are the comparison points of $x$ and $y$, i.e., if $x$ belongs to the side joining $x_i$ and $x_j$, then $\overline{x}$ belongs to the side joining  $\overline{x}_i$ and $\overline{x}_j$ such that $d(x_i,x) = d_{M^2_{\kappa}}(\overline{x}_i,\overline{x})$.

A {\it $\CAT(\kappa)$ space} is geodesic space where every geodesic triangle satisfies the CAT$(\kappa)$ inequality. $\CAT(\kappa)$ spaces are also known as spaces of {\it curvature bounded above} by $\kappa$ (in the sense of Alexandrov). In any $\CAT(\kappa)$ space there exists a unique geodesic joining each pair of points. 

An {\it ${\R}$-tree} is a uniquely geodesic space $X$ such that if $x,y,z \in X$ with $[y,x]\cap [x,z]=\{x\}$, then $[y,x]\cup [x,z]=[y,z]$. It is easily seen that a metric space is an ${\R}$-tree if and only if it is a $\CAT(\kappa)$ space for any real $\kappa \le 0$. 

Another geometric condition that plays an essential role in the sequel is that of $\delta$-hyperbolicity. There are several ways to introduce this concept and we follow here the one attributed to Rips (see \cite[p. 399]{Bri99}). Given $M \ge 0$, a geodesic triangle in a metric space is called {\it $M$-slim} if any of its sides is contained in the $M$-neighborhood of the union of the other two sides. A geodesic space $X$ is called {\it $\delta$-hyperbolic} for some $\delta\geq 0$ if every geodesic triangle in it is $\delta$-slim. If a geodesic space is $\delta$-hyperbolic for some $\delta \ge 0$, then it is also said to be {\it Gromov hyperbolic}. $\CAT(\kappa)$ spaces  with $\kappa <0$  are {\it $\delta$-hyperbolic}, where $\delta$ only depends on $\kappa$. Moreover, a geodesic space is an $\R$-tree if and only if it is $0$-hyperbolic. Note also that there exist $\CAT(0)$ spaces which are not $\delta$-hyperbolic such as Hilbert spaces. 
 
Given three points $x,y,z$ in a metric space, the {\it Gromov product} $(y|z)_x$ is the nonnegative number defined by
$$
(y|z)_x=\dfrac{1}{2}\left(d(x,y)+d(x,z)-d(y,z)\right).
$$
The following characterization of $\delta$-hyperbolicity is due Gromov and often used as an alternative definition (see \cite[Lemma 1.2.3 and Exercise 1.2.4]{BuySch07}).

\begin{proposition}\label{prop-gromov-charact}
A geodesic space $(X,d)$ is $\delta$-hyperbolic for some $\delta\geq 0$ if and only if there exists $\delta' \ge 0$ such that for all three points
$x,y,z \in X$, fixing any geodesic segments $[x,y]$ and $[x,z]$ joining $x$ and $y$, and $x$ and $z$, respectively, the following implication holds: if $y' \in [x,y]$ and $z' \in [x,z]$ are such that $d(x,y')=d(x,z')\leq (y|z)_x$, then $d(y',z') \leq \delta'$.
\end{proposition}

\section{Geodesic rays and the fixed point property}\label{section-rays}

Let $(X,d)$ be a metric space and $A \subseteq X$. A mapping $T : A \to X$ is called {\it nonexpansive} if $d(Tx, Ty) \le d(x,y)$ for all $x,y \in A$. We say that $A$ has the {\it fixed point property} (FPP for short) if each nonexpansive mapping $T : A\to A$ has at least one fixed point, i.e., a point $x\in A$ such that $Tx=x$. A very well-known result from 1965 proved independently by Browder \cite{Bro65}, G\"{o}hde \cite{Goh65} and Kirk \cite{Kir65} says that every closed, convex and bounded subset of a Hilbert space has the FPP. In 1980, Ray \cite{Ray80} approached the converse problem and proved that boundedness is a necessary condition for a closed and convex subset of a Hilbert space to have the FPP. Similar results in the setting of Banach spaces can be found, e.g., in \cite{Dom12, Rei76, Rei80, Tak10}.

After the publication of the papers \cite{Kir03, Kir04} due to Kirk, geodesic metric spaces have called the attention of many authors working in metric fixed point theory. Especially relevant in the study of the FPP proved to be the existence of upper bounds on the curvature in the sense of Alexandrov. Since Hilbert spaces are the only Banach spaces which are $\CAT(0)$, it was natural to consider the question whether the Browder-G\"{o}hde-Kirk theorem and Ray's result mentioned before hold true in complete $\CAT(0)$ spaces. Regarding the Browder-G\"{o}hde-Kirk theorem, the answer is positive (see \cite{Kir03}), however Ray's result fails as there exist broad classes of $\CAT(0)$ spaces where a closed and convex set has the FPP if and only if it is geodesically bounded (and hence not necessarily bounded). Such examples include the complex Hilbert ball with the hyperbolic metric (see \cite[Theorems 25.2, 32.2]{GoeRei84} and \cite[Corollary 4.4]{Pia15}), complete ${\mathbb R}$-trees (see \cite[Theorem 4.3]{EspKir06}), or even complete $\CAT(\kappa)$ spaces with $\kappa < 0$ (see \cite[Corollary 4.2]{Pia15}). In fact, this characterization of the FPP in terms of geodesic boundedness holds in the setting of complete $\CAT(0)$ spaces that are additionally $\delta$-hyperbolic (see \cite[Corollary 3.2]{Pia18}). Other results related to the FPP of unbounded sets in geodesic spaces can be found in \cite{EspPia13, Pia17, Pia15}.

Another more general property which was considered in this line is defined as follows: we say that $A$ has the {\it approximate fixed point property} (AFPP for short) if $\inf\{d(x,Tx) : x\in A\}=0$ for every nonexpansive mapping $T : A\to A$. It is immediate that every closed, convex and bounded subset of a Banach space has the AFPP (see \cite[Lemma 3.1]{GoeKir90}). However, there exist unbounded, closed and convex sets that have the AFPP. Reich \cite{Rei83} showed that a closed and convex subset of a reflexive Banach space has the AFPP if and only if it is linearly bounded. Shafrir \cite{Sha90} used the notion of directional curve to characterize the AFPP of convex sets in a class of metric spaces which includes, in particular, Banach spaces or complete Busemann convex geodesic spaces. 

\begin{definition}
Let $(X,d)$ be a metric space. A curve $\gamma:[0,\infty)\to X$ is said to be {\it directional} if there exists $b\geq 0$ such that
$$
|s-t|-b\le d(\gamma(s),\gamma(t))\le |s-t|,
$$
for all $s,t\geq 0$. 

A subset of $X$ is called {\it directionally bounded} if it contains no directional curves.

A sequence $(x_n)$ in $X$ is said to be {\it directional} if the following two conditions hold:
\begin{enumerate}
\item $\displaystyle\lim_{n\to \infty} d(x_0,x_n) = \infty$;
\item  there exists $b\geq 0$ such that
\begin{equation}\label{dir-seq-cond2}
d(x_{n_1},x_{n_l}) \geq \sum_{i=1}^{l-1} d(x_{n_i},x_{n_{i+1}}) - b,
\end{equation}
for all $n_1 < n_2 <\cdots < n_l $.
\end{enumerate}
\end{definition}

Clearly, every geodesic ray is a directional curve, so directionally bounded sets are always geodesically bounded. Note also that a convex subset of a geodesic space is directionally bounded if and only if it does not contain any directional sequence (see \cite[Lemma 2.3]{Sha90}).

In Banach spaces or in complete Busemann convex spaces, a convex set has the AFPP if and only if it is directionally bounded (see \cite[Theorem 2.4]{Sha90}). Moreover, the directional boundedness can also be used to give a characterization of reflexivity in Banach spaces. Namely, a Banach space is reflexive if and only if every closed and convex subset of it that is linearly bounded is directionally bounded (see \cite[Proposition 3.5]{Sha90}). In the nonlinear case we have the following result (a corresponding one for the case of a Busemann convex space that is additionally $\delta$-hyperbolic is given in Proposition \ref{thm-geod-dir-bd-Busemann}).

\begin{proposition}\label{thm-geod-dir-bd}
If $(X,d)$ is a complete $\CAT(0)$ space, then every closed and convex subset of $X$ that is geodesically  bounded is directionally bounded.
\end{proposition}
\begin{proof} Let $A$ be a closed and convex subset of $X$. We show that if $A$ is not directionally bounded, then it is not geodesically bounded either. Take $(x_n)$ a directional sequence in $A$ with constant $b$. For $n \in \N$, denote $d_n = d(x_0,x_n)$. Then $\lim_{n \to \infty} d_n = \infty$ and, by \eqref{dir-seq-cond2}, for $m,n \in \N$ with $0 < m < n$ we have $d_n \ge  d_m + d(x_m,x_n) - b$.

Let $m,n \in \N$ with $0 < m < n$ such that $d_m,d_n > 0$. For the geodesic triangle $\Delta(x_0,x_m,x_n)$, consider a comparison triangle $\overline{\Delta}(\overline{x}_0,\overline{x}_m,\overline{x}_n)$ in $\R^2$ and denote its interior angle at $\overline{x}_0$ by $\overline{\alpha}_{m,n}$. Let $b_{m,n}=d_m + d(x_m,x_n) - d_n \in [0,b]$. The cosine law in $\R^2$ yields
\begin{align*}
b^2_{m,n} + d^2_m + d^2_n -2d_mb_{m,n}   - 2d_md_n + 2d_nb_{m,n}  &= (b_{m,n} - d_m + d_n)^2\\
& = d(x_m,x_n)^2 = d^2_m + d^2_n - 2d_md_n\cos \overline{\alpha}_{m,n},
\end{align*}
from where
\[b^2 +  2d_nb \ge b^2_{m,n} +  2d_nb_{m,n}  \ge b^2_{m,n} +  2d_nb_{m,n} - 2d_mb_{m,n} = 2d_md_n(1-\cos \overline{\alpha}_{m,n}),\]
so
\[\sin^2\frac{\overline{\alpha}_{m,n}}{2}\leq \frac{b}{2d_m}\left(\frac{b}{2d_n} + 1\right).\]

For $n,k \ge 1$, if $d_n \ge k$, pick $x_k^n \in [x_0,x_n]$ so that $d(x_0,x_k^n) = k$. Let $k \ge 1$ and $\varepsilon > 0$. Then there exists $n_k \ge 1$ such that for all $n,m \ge n_k$, $d_m,d_n \ge k$ and $\overline{\alpha}_{m,n} < \varepsilon$. If $n,m \ge n_k$, denote $\overline{x}_k^m$ and $\overline{x}_k^n$ the comparison points in $\overline{\Delta}(\overline{x}_0,\overline{x}_m,\overline{x}_n)$ of $x_k^m$ and $x_k^n$, respectively. As $\left(\overline{x}_k^n\right)_{n\ge n_k}$ is Cauchy and $d(x_k^m,x_k^n) \le d(\overline{x}_k^m,\overline{x}_k^n)$ for all $n,m \ge n_k$, it follows that $\left(x_k^n\right)_{n\ge n_k}$ is Cauchy too, hence it converges to some $x_k^* \in A$.

Let $k,l \ge 1$ with $k < l$. For $n$ sufficiently large, $x_k^n = (1-k/l)x_0+(k/l) x_l^n$. By Busemann convexity,
\[d(x_k^n, (1-k/l)x_0+(k/l) x_l^*) \le \frac{k}{l}d(x_l^n,x_l^*),\]
which shows that $x_k^* = (1-k/l)x_0+(k/l) x_l^*$, so $x_k^* \in [x_0,x_l^*]$. Finally, $\bigcup_{k}[x_0,x_k^*] \subseteq A$ is a geodesic ray.
\end{proof}

The following notions will play an important role in the proof of our main result.

\begin{definition}
Let $(X,d)$ be a metric space, $\lambda \ge 1$, $\eps \ge 0$, and $k > 0$. A {\it $(\lambda,\varepsilon)$-quasi-geodesic} is a mapping $\gamma : [a,b] \subseteq \R \to X$ such that
\begin{equation}\label{def-quasi-geod-eq}
\frac{1}{\lambda}|s-t| - \varepsilon \le d(\gamma(s),\gamma(t))\le \lambda |s-t| + \varepsilon,
\end{equation}
for all $s,t\in [a,b]$. We say that $\gamma$ {\it joins} $\gamma(a)$ and $\gamma(b)$. If instead of the interval $[a,b]$ one considers $[0,\infty)$, then $\gamma$ (or its image) is called a {\it $(\lambda,\varepsilon)$-quasi-geodesic ray}. If $\gamma$ is a $(\lambda,\varepsilon)$-quasi-geodesic for every $\varepsilon > 0$, then we simply say that $\gamma$ is a {\it $\lambda$-quasi-geodesic}. At the same time, a $\lambda$-quasi-geodesic defined on $[0,\infty)$ is called a {\it $\lambda$-quasi-geodesic ray}.

A {\it $k$-local $(\lambda,\varepsilon)$-quasi-geodesic} is a mapping $\gamma: [a,b] \subseteq \R\to X$ such that \eqref{def-quasi-geod-eq} holds for all $s,t\in [a,b]$ with $|s-t| \le k$. The ``$k$-local'' versions of the other notions from above are defined in a similar way.
\end{definition}

It is clear that every geodesic is a $\lambda$-quasi-geodesic for any $\lambda \ge 1$. Furthermore, directional curves with constant $b$ are $(1,b)$-quasi-geodesic rays. However, not all quasi-geodesic rays are directional curves (to see this, one can consider the curve given in \cite[p. 142, Exercise 8.23]{Bri99}). 

For our main result, the question whether a local quasi-geodesic ray is actually, up to an appropriate change of constants, a quasi-geodesic ray is particularly relevant. Although a positive answer can be anticipated from \cite[p. 407, Remark]{Bri99}, for completeness we clarify this aspect in Proposition \ref{prop-localray} where we use a proof strategy similar to the one of \cite[p. 405, Theorem 1.13]{Bri99}.

Before stating this result, we recall that given a metric space $(X,d)$, a {\it $(\lambda,\varepsilon)$-quasi-geodesic triangle} $\Delta = \Delta(x_1,x_2,x_3)$ in $X$, where $\lambda \ge 1$ and $\eps \ge 0$, consists of three points $x_1, x_2,x_3 \in X$ (its {\it vertices}) and the images of three $(\lambda,\varepsilon)$-quasi-geodesics (its {\it sides}) joining each pair of points. A $(\lambda,\varepsilon)$-quasi-geodesic triangle is called {\it $M$-slim}, where $M \ge 0$, if each of its sides is contained in the $M$-neighborhood of the union of the other two sides. As before, one can also consider the notion of $\lambda$-quasi-geodesic triangle and these are in fact the triangles that we will work with. The next property follows from \cite[p. 402, Corollary 1.8]{Bri99}.

\begin{remark}\label{rmk-quasi-slim}
If $X$ is a $\delta$-hyperbolic geodesic space, then for every $\lambda \ge 1$ there exists $M = M(\delta,\lambda)$ such that all $\lambda$-quasi-geodesic triangles are $M$-slim. 
\end{remark}

\begin{proposition}\label{prop-localray}
Let $(X,d)$ be a $\delta$-hyperbolic geodesic space, $\lambda \ge 1$, and $M=M(\delta,\lambda)$ be given by Remark \ref{rmk-quasi-slim}. If $\gamma$ is a $k$-local $\lambda$-quasi-geodesic ray such that $k > 8\lambda M$, then $\gamma$ is a $(\lambda^*,\varepsilon)$-quasi-geodesic ray, where  
\[\lambda^* = \left(\frac {1}{\lambda} - \frac{4M}{k/2+\lambda M}\right)^{-1} \quad \text{and} \quad \varepsilon = 2M.\]
\end{proposition}
\begin{proof} 
One can easily see that $\lambda^* \ge 1$. To prove the result, it is enough to show that for every $a,b \in [0,\infty)$, $\left.\gamma\right|_{[a,b]}$ is a $(\lambda^*,\varepsilon)$-quasi-geodesic. Moreover, we can assume that $b-a > k$ (otherwise the conclusion is immediate because $\lambda \le \lambda^*$). Fix a geodesic segment joining $\gamma(a)$ and $\gamma(b)$ and denote it by $[\gamma(a), \gamma(b)]$.
\begin{claim}\label{thm-localray-cl1}
The image of $\left.\gamma\right|_{[a,b]}$ is contained in a $2M$-neighborhood of $[\gamma(a), \gamma(b)]$. 
\end{claim}
\begin{proof}[Proof of Claim \ref{thm-localray-cl1}]
Let $x = \gamma(t)$, where $t \in [a,b]$, such that 
\[\dist(x,[\gamma(a), \gamma(b)])  = \max_{s\in [a,b]} \dist(\gamma(s), [\gamma(a), \gamma(b)]).\]
We know that $\max\{t-a,b-t\} > 4\lambda M$ (since $b-a > k > 8\lambda M)$. Depending on the values of $t-a$ and $t-b$, we distinguish two situations.

Case I : $t-a, b-t > 4 \lambda M$.

Take $t_1 , t_2 \in [a,b]$ such that $t-t_1, t_2 -t \in (4\lambda M, k/2)$. Note that $t \in (t_1,t_2)$. Furthermore, $t_2 -t_1 < k$  and so $\left.\gamma\right|_{[t_1,t_2]}$ is a $\lambda$-quasi geodesic. Denote $y=\gamma(t_1)$, $z=\gamma(t_2)$, $y' \in P_{[\gamma(a), \gamma(b)]}(y)$, and $z' \in P_{[\gamma(a), \gamma(b)]}(z)$. 
Then 
\begin{equation}\label{thm-localray-eq1}
d(x,y) = d(\gamma(t), \gamma(t_1)) \geq \frac{1}{\lambda}(t-t_1) > 4M.
\end{equation}  
Similarly, $d(x,z) > 4M$. Consider two geodesic segments $[y,y']$ and $[y',z]$ joining $y$ and $y'$, and $y'$ and $z$, respectively. Form the $\lambda$-quasi-geodesic triangle $\Delta(y',y,z)$ whose sides are $[y,y']$, $[y',z]$ and the image of the $\lambda$-quasi-geodesic $\left.\gamma\right|_{[t_1,t_2]}$.
Then $\Delta(y',y,z)$ is $M$-slim, so there exists $u\in [y,y'] \cup [y',z]$ such that $d(x,u) \leq  M$.

Suppose $u\in [y',z]$. Take a geodesic segment $[z,z']$ joining $z$ and $z'$ and the geodesic segment $[y',z'] \subseteq [\gamma(a),\gamma(b)]$ joining $y'$ and $z'$. The geodesic triangle $\Delta(y',z,z')$ whose sides are $[y',z]$, $[z,z']$ and $[y',z']$ is $M$-slim, so there exists $w\in [z,z'] \cup [y',z']$ such that
$d(u,w) \leq M$.
 
Thus, there exists $p \in [y,y'] \cup [z,z'] \cup [y',z']$ such that
\begin{equation}\label{thm-localray-eq2}
d(x,p) \leq 2M.
\end{equation} 
Now we show that actually $p\in  [y',z']$. If $p\in  [y, y']$, we have 
\begin{align*}
d(x, y') &\le d(x,p) + d(p,y') = d(x,p) + d(y,y') - d(y,p) \\
& \le d(x,p) + d(y,y') - (d(x,y) - d(x,p))\\
& =  2d(x,p) + d(y,y') - d(x,y) < d(y,y') \quad \text{by }\eqref{thm-localray-eq2}\text{ and } \eqref{thm-localray-eq1}.
\end{align*}
This contradicts the choice of $x$. In a similar way one obtains again a contradiction if $p\in  [z, z']$.

Case II: $t-a > 4 \lambda M$ and $ b-t \le 4 \lambda M$. (The case $b-t > 4 \lambda M$ and $ t-a\leq 4 \lambda M$ is dealt with in a similar way.)
 
Take  $t_1 \in [a,b]$ such that $t-t_1 \in (4 \lambda M, k/2)$. Note that $t \in (t_1,b]$. Furthermore, 
\[b-t_1 = b - t + t - t_1 < 4 \lambda M + \frac{k}{2} < k,\] 
so $\left.\gamma\right|_{[t_1,b]}$ is a $\lambda$-quasi-geodesic. Denote $y=\gamma(t_1)$ and $y' \in P_{[\gamma(a), \gamma(b)]}(y)$. Then
\[d(x,y) = d(\gamma(t),\gamma(t_1)) \ge \frac{1}{\lambda}|t-t_1|> 4M.\]

Consider a geodesic segment $[y,y']$ joining $y$ and $y'$ and the geodesic segment $[y',\gamma(b)] \subseteq [\gamma(a),\gamma(b)]$ joining $y'$ and $\gamma(b)$. Form the $\lambda$-quasi-geodesic triangle $\Delta(y',y,\gamma(b))$ whose sides are $[y,y']$, $[y',\gamma(b)]$ and the image of the $\lambda$-quasi-geodesic $\left.\gamma\right|_{[t_1,b]}$. This triangle is $M$-slim, so there exists $p\in [y,y'] \cup [y',\gamma(b)]$ such that $d(x,p) \le M$. In fact, $p \in [y',\gamma(b)]$ because if $p\in [y, y']$ we get a contradiction as in the previous case.

Hence, in both cases, we find $p\in [\gamma(a),\gamma(b)]$ such that $d(x,p) \le 2M$. This finishes the proof of the claim.
\end{proof}

Now we show that $\left.\gamma\right|_{[a,b]}$ is a $(\lambda^*,\eps)$-quasi-geodesic. Dividing the $k$-local $\lambda$-quasi-geodesic $\left.\gamma\right|_{[a,b]}$ into sufficiently small subpaths and using the triangle inequality we get 
\[d(\gamma(s), \gamma(t))  \le \lambda |s - t| \le \lambda^* |s - t|,\]
for all $s, t \in [a,b]$. Thus, we only need to prove the left-hand inequality in the definition of a $(\lambda^*,\eps)$-quasi-geodesic. Bearing in mind that a subpath of a $k$-local $\lambda$-quasi-geodesic is a $k$-local $\lambda$-quasi-geodesic as well, it is enough to prove that 
\[\frac{1}{\lambda^*}(b-a) - \varepsilon  \le d(\gamma(a),\gamma(b)).\]

Let $n = \lfloor (b-a)/(k/2 + \lambda M)\rfloor$. Note that $n \ge 1$ as $b-a > k > k/2 + \lambda M$. Take $t_0 = a$ and for $i \in \{0,\ldots, n-1\}$, let 
\[t_{i+1} = t_i + \frac{k}{2}+\lambda M \in [a,b].\] 
Then $\left.\gamma\right|_{[t_i ,t_{i+1}]}$ is a $\lambda$-quasi-geodesic for all $i \in \{0,\ldots, n-1\}$. Likewise, $\left.\gamma\right|_{[t_n ,b]}$ is a $\lambda$-quasi-geodesic.

\begin{claim}\label{thm-localray-cl2}
Suppose $n \ge 2$, fix $i \in \{0,\ldots, n-2\}$, denote $x=\gamma(t_i)$, $m=\gamma(t_{i+1})$, $y=\gamma(t_{i+2})$, and take the corresponding projections onto $[\gamma(a),\gamma(b)]$
\[x' \in P_{[\gamma(a),\gamma(b)]}(x), \quad m' \in P_{[\gamma(a),\gamma(b)]}(m), \quad y' \in P_{[\gamma(a),\gamma(b)]}(y).\]
Then $m'\in [x',y'] \subseteq [\gamma(a), \gamma(b)]$. 
\end{claim}
\begin{proof}[Proof of Claim \ref{thm-localray-cl2}]
By Claim \ref{thm-localray-cl1}, $d(x,x') \le 2 M$. Additionally,
\[d(x,m) \ge \frac{1}{\lambda}\left(\frac{k}{2}+\lambda M\right) > 5M.\]
Hence, 
\[5M < d(x,m) \le d(x,x') + d(x',m) \le 2M + d(x',m),\] 
so $d(x',m) > 3M$. Similarly, $d(y',m) > 3M$.\\

Denote $x_0  = \gamma(t_i + \lambda M)$ and $ y_0 = \gamma(t_{i+1} + k/2)$. Take four geodesic segments $[x',x]$, $[x',x_0]$, $[x',m]$ and $[x',y_0]$ that join $x'$ and the respective points $x$, $x_0$, $m$ and $y_0$. Let also $[y',y_0]$ be a geodesic segment joining $y'$ and $y_0$. Consider the $\lambda$-quasi-geodesic triangle $\Delta(x',x_0,y_0)$ of sides $[x',x_0]$, $[x',y_0]$ and the image of $\left.\gamma\right|_{[t_i + \lambda M,t_{i+1} + k/2]}$. Let also $\Delta(x',y_0 ,y')$ be the geodesic triangle of sides $[x',y_0]$, $[y',y_0]$ and $[x',y']$. Since $\Delta(x',x_0,y_0)$ is $M$-slim, there exists $u \in [x',x_0] \cup [x',y_0]$ such that $d(m,u) \le M$. If $u \in [x',y_0]$, then, as $\Delta(x',y_0 ,y')$ is $M$-slim, there exists $w \in [y',y_0] \cup [x',y']$ such that $d(u,w) \le M$. Hence, one can find $p \in [x',x_0] \cup [y',y_0] \cup [x',y']$ so that $d(m,p) \le 2M$. 

Suppose $p \in [x',x_0]$. Consider the $M$-slim $\lambda$-quasi-geodesic triangle $\Delta(x',x_0,x)$ of sides $[x',x_0]$, $[x',x]$ and the image of $\left.\gamma\right|_{[t_i,t_{i} + \lambda M]}$. If $d(p,u) \le M$ for some $u \in [x',x]$, then $d(p,x) \le d(p,u) + d(u,x) \le M + 2M = 3M$. But then
\[5M < d(m,x) \le d(m,p) + d(p,x) \le 2M + 3M = 5M,\]
a contradiction. If $d(p,u) \le M$ for some $u=\gamma(t_i + \alpha)$, where $\alpha \in [0,\lambda M]$, then 
\[d(u,m) = d(\gamma(t_i + \alpha),\gamma(t_i+ k/2 + \lambda M) \ge \frac{1}{\lambda}\left(\frac{k}{2} + \lambda M - \alpha\right) \ge \frac{k}{2\lambda} > 4M.\]
However, $d(u,m) \le d(u,p) + d(p,m) \le M + 2M = 3M$, which is again a contradiction. This shows that $p \notin [x',x_0]$. Similarly, $p \notin [y',y_0]$, so $p \in [x',y']$.

Let $[m,m']$ and $[m,p]$ be two geodesic segments joining $m$ and $m'$, and $m$ and $p$, respectively. The geodesic triangle $\Delta(m,m',p)$ whose sides are $[m,m']$, $[m,p]$ and $[m',p] \subseteq [\gamma(a),\gamma(b)]$ is $M$-slim. Recall that, by Claim \ref{thm-localray-cl1}, $d(m,m') \leq  2M$. Thus, 
$d(u,m) \leq 3M$ for all $u\in [m',p]$. This shows that $x',y' \notin [m',p]$ and, as $p \in [x',y']$, we must have $m' \in [x', y']$.
\end{proof}

Denote now $p_i \in P_{[\gamma(a),\gamma(b)]} \gamma(t_i )$ for $i\in \{0,\cdots, n\}$. Then 
\[\left(\frac{k}{2}+\lambda M\right) \frac{1}{\lambda} \le d(\gamma(t_i ),\gamma(t_{i+1})) \le d(\gamma(t_i ),p_i) + d (p_i, p_{i+1}) + d(p_{i+1},\gamma(t_{i+1})) \le 4M + d(p_i,p_{i+1}),\]
and hence, 
\[d(p_i,p_{i+1}) \ge \frac{k}{2 \lambda}- 3M,\] 
for all $i\in \{0,\cdots, n-1\}$. Likewise, 
\[d(p_n ,\gamma(b)) \ge d(\gamma(t_n),\gamma(b)) - d(p_n,\gamma(t_n)) \ge \frac{1}{\lambda}( b-t_n) -2M.\]
Therefore, using Claim \ref{thm-localray-cl2},
\begin{align*}
d(\gamma(a),\gamma(b)) &= \sum_{i=0}^{n-1}d(p_i,p_{i+1}) + d (p_n,\gamma(b)) \ge n\left(\frac{k}{2 \lambda}- 3M\right) +  \frac{1}{\lambda}( b-t_n) - 2M\\
& = \frac{1}{\lambda}\left(n\frac{k}{2} + b-t_n\right) - 3nM - 2M = \frac{b-a}{\lambda} - 4nM - 2M\\
& \ge \frac{b-a}{\lambda} - 4M\frac{b-a}{k/2 + \lambda M} - 2M = \left(\frac{1}{\lambda} - \frac{4M}{k/2+\lambda M}\right)(b-a) - 2M = \frac{1}{\lambda^*}(b-a) - \eps.
\end{align*} 
 \end{proof}
 
We see next that if the $\delta$-hyperbolic geodesic space is also Busemann convex, then the existence of a quasi-geodesic ray yields the existence of a geodesic ray. Combined with Proposition \ref{prop-localray}, this shows that for the existence of a geodesic ray, the presence a local quasi-geodesic ray is sufficient.
 
\begin{proposition}\label{prop-quasi-ray}
Let $(X,d)$ be a complete Busemann convex space that is additionally $\delta$-hyperbolic and $A \subseteq X$ closed and convex. If $A$ contains a $(\lambda,\eps)$-quasi-geodesic ray, where $\lambda \ge 1$ and $\eps \ge 0$, then it contains a geodesic ray.
\end{proposition}
\begin{proof}
Let $\gamma : [0,\infty) \to A$ be a $(\lambda,\eps)$-quasi-geodesic ray. Then 
\[\frac{1}{\lambda}|s-t| - \varepsilon \le d(\gamma(s),\gamma(t))\le \lambda |s-t| + \varepsilon,\]
for all $s,t\in [0,\infty)$.

Take $\alpha > 1$ such that 
\[\beta = \frac{1}{\lambda} + \lambda + \alpha\left(\frac{1}{\lambda} - \lambda\right) > 0.\]
Let $x_0 = \gamma(0)$ and for $n \ge 1$, define $x_n =\gamma(\alpha^n)$. Then 
\[2(x_n | x_{n+1})_{x_0}  \ge   \frac{1}{\lambda}(\alpha^n + \alpha^{n+1}) - 2\eps - \lambda(\alpha^{n+1}-\alpha^n) - \eps = \beta\alpha^n - 3\eps.\]
Thus, for $n$ large enough, $(x_n | x_{n+1})_{x_0} \ge \beta \alpha^n/4$. Observe also that $\lim_{n \to \infty}d(x_0,x_n) = \infty$.

For $n,k \ge 1$, if $d(x_0,x_n) \ge k$, pick $x_k^n \in [x_0,x_n]$ so that $d(x_0,x_k^n) = k$. Fix $k \ge 1$. For $n$ sufficiently large, $d(x_0,x_n), d(x_0,x_{n+1}) \ge k$ and $(x_n | x_{n+1})_{x_0} \ge k$. Moreover, as $(x_n | x_{n+1})_{x_0} \le \min\{d(x_0,x_n), d(x_0,x_{n+1})\}$, we can take $y_n \in [x_0,x_n]$ and $y_{n+1} \in [x_0,x_{n+1}]$ such that $d(x_0,y_n) = d(x_0,y_{n+1}) = (x_n | x_{n+1})_{x_0}$. By Busemann convexity, 
\[d(x_k^n, x_k^{n+1}) \le \frac{k}{(x_n | x_{n+1})_{x_0}}d(y_n,y_{n+1}).\]
Applying Proposition \ref{prop-gromov-charact}, we conclude that there exists $\delta^* >0$ such that
\[d(x_k^n, x_k^{n+1})  \le \frac{k\delta^*}{(x_n | x_{n+1})_{x_0}} \le \frac{4k\delta^*}{\beta \alpha^n},\]
for all $n$ sufficiently large. So the sequence $(x_k^n)_n$ is Cauchy and hence converges to a point $x_k^*\in A$ which satisfies $d(x_0, x_k^*) =k$. 

Let $k,l \ge 1$ with $k < l$. As in the proof of Proposition \ref{thm-geod-dir-bd} one can use Busemann convexity to show that $x_k^* \in [x_0,x_l^*]$, hence $\bigcup_{k}[x_0,x_k^*]$ is a geodesic ray in $A$.
\end{proof} 

\begin{corollary}\label{cor-local-quasi-ray}
Let $(X,d)$ be a complete Busemann convex space that is additionally $\delta$-hyperbolic and $A \subseteq X$ closed and convex. Suppose $\lambda \ge 1$ and $M=M(\delta,\lambda)$ is given by Remark \ref{rmk-quasi-slim}. If $A$ contains a $k$-local $\lambda$-quasi-geodesic ray such that $k > 8\lambda M$, then it contains a geodesic ray.
\end{corollary}

Since any directional curve is a quasi-geodesic ray, we can apply Proposition \ref{prop-quasi-ray} to get the following analogue of Proposition \ref{thm-geod-dir-bd}.
\begin{proposition}\label{thm-geod-dir-bd-Busemann}
In a complete Busemann convex space that is additionally $\delta$-hyperbolic, every closed and convex set that is geodesically  bounded is directionally bounded.
\end{proposition}

\begin{remark}
Let us notice that the previous result is in general not true if the space is merely assumed to be Busemann convex. Actually, it is well-known that every separable Banach space has an equivalent strictly convex norm (see, e.g., \cite[p. 60, Theorem 1.5]{ADL97}). In particular, one can renorm $\ell_1$ to make it Busemann convex. Since this space is not reflexive, according to \cite[Proposition 3.5]{Sha90}, geodesic boundedness cannot imply directional boundedness for every closed and convex set. 
\end{remark}

In contrast to $\delta$-hyperbolic Busemann convex spaces, in Hilbert spaces, the existence of a quasi-geodesic ray in a closed and convex set does not yield the existence of a geodesic ray as the following example shows. (Recall however that, by \cite[Proposition 3.5]{Sha90}, in any Hilbert space, the existence of a directional curve implies the existence of a geodesic ray.)

\begin{example} Let $A \subseteq \ell_2$ be given by
\[A = \left\{x = (x_1,x_2,x_3, \ldots ) \in \ell_2: 0 \le x_n \le 10^n \text{ for all } n \ge 1\right\}.\]
Then $A$ is closed, convex and linearly bounded, hence geodesically bounded. We construct next a quasi-geodesic ray in $A$. To this end, consider first the sequence $(x^k)$ in $A$, where $x^0 = (0,0,0, \ldots)$ and  
$$
x^k_n=\left\{
\begin{array}{ll}
10^n & \mbox{if } n \le k \\
0 & \mbox{otherwise},
\end{array}\right.
$$
for all $k \ge 1$.
Let $a_0 = 0$ and $a_k = \sum_{n =1}^k 10^n$ for $k \ge 1$. Note that $a_{k+1} - a_k = 10^{k+1} = \left\|x^{k+1}-x^k\right\|_2$ for all $k \in \N$. Define now $\gamma : [0,\infty) \to A$,
\[\gamma(t) = \left(1- \frac{t-a_k}{10^{k+1}}\right)x^k + \frac{t-a_k}{10^{k+1}}x^{k+1}, \quad \text{for } a_k \le t < a_{k+1}, \text{ where } k \in \N.\]
Note that $\gamma(a_k) = x^k$ for all $k \in \N$. We show that $\gamma$ is a $\sqrt{11/3}$-quasi-geodesic ray. 

Case I: Let $k \ge 1$ and $a_{k-1} \le s \le t < a_k$. Then $\left\|\gamma(t) - \gamma(s)\right\|_2 = t-s$.

Case II: Let $k \ge 1$, $a_{k-1} \le s < a_k$, and $a_{k} \le t < a_{k+1}$. Denote
\[u = \left\|x^k - \gamma(s)\right\|_2 = a_k - s \quad \text{and} \quad v = \left\|\gamma(t) - x^k\right\|_2 = t - a_k.\]

Then
\[t-s = u + v \le  \sqrt{2}\sqrt{u^2+v^2} = \sqrt{2} \left\|\gamma(t) - \gamma(s)\right\|_2\]
and
\[\left\|\gamma(t) - \gamma(s)\right\|_2 \le \left\|\gamma(t) - x^k\right\|_2 + \left\|x^k - \gamma(s)\right\|_2 = v+u= t-s.\]

Case III: Let $1 \le k < l$, $a_{k-1} \le s < a_k$, and $a_{l} \le t < a_{l+1}$. Denote
\[u = \left\|x^k - \gamma(s)\right\|_2 = a_k - s, \qquad v = \left\|\gamma(t) - x^l\right\|_2 = t - a_l,\]
and
\begin{align*}
w & =  \left\|x^l - x^k\right\|_2 = \sqrt{\sum_{i=k}^{l-1}10^{2(i+1)}} = \frac{1}{3\sqrt{11}}\sqrt{10^{2(l+1)} - 10^{2(k+1)}}\\
& \ge \frac{1}{3\sqrt{11}}\left(10^{l+1} - 10^{k+1}\right) = \frac{3}{\sqrt{11}}\sum_{i=k}^{l-1} 10^{i+1}= \frac{3}{\sqrt{11}}\sum_{i=k}^{l-1} \left\|x^{i+1} - x^i\right\|_2.
\end{align*}

Then
\begin{align*}
t-s & = a_k - s  + \sum_{i = k}^{l-1}(a_{i+1}-a_i) +  t - a_l = u + v + \sum_{i=k}^{l-1} \left\|x^{i+1} - x^i\right\|_2\\
& \le u + v + \frac{\sqrt{11}}{3}w \le \frac{\sqrt{11}}{3} (u+v+w) \le \sqrt{\frac{11}{3}}\sqrt{u^2+v^2+w^2} = \sqrt{\frac{11}{3} }\left\|\gamma(t) - \gamma(s)\right\|_2,
\end{align*}
where the last inequality follows from $3(u^2+v^2+w^2) \ge (u+v+w)^2$.

At the same time,
\begin{align*}
\left\|\gamma(t) - \gamma(s)\right\|_2 & \le \left\|\gamma(t) - x^l\right\|_2 + \left\|x^k - \gamma(s)\right\|_2 + \sum_{i=k}^{l-1} \left\|x^{i+1} - x^i\right\|_2 \\
& = t - a_l + a_k - s + \sum_{i=k}^{l-1} (a_{i+1} - a_i) = t-s.
\end{align*}

\end{example}

\section{Geodesic rays and the Lion-Man game}\label{lion-game}

Let $(X,d)$ be a uniquely geodesic space and $A \subseteq X$ nonempty and convex. Take $D > 0$ and suppose that $L_0, M_0 \in A$ are the starting points of the lion and the man, respectively. At step $n+1$, $n \in \N$, the lion moves from the point $L_n$ to the point $L_{n+1} \in [L_n,M_n]$ such that $d(L_n,L_{n+1}) = \min\{D,d(L_n,M_n)\}$. The man moves from the point $M_n$ to the point $M_{n+1} \in A$ satisfying $d(M_n,M_{n+1}) \le D$. We say that the lion wins if the sequence $(d(L_{n+1},M_n))$ converges to $0$. Otherwise the man wins. Denote in the sequel $D_n = d(L_n,M_n)$, $n \in \N$.

It is easy to see that the lion wins if and only if either of the following two mutually exclusive situations holds:
\begin{itemize}
\item[(1)] there exists $n_0 \in \N$ such that $D_{n_0} \le D$. In this case, $L_{n+1} = M_n$ for all $n \ge n_0$;
\item[(2)] $D_n > D$ for all $n \in \N$ and $\lim_{n \to \infty}D_n = D$. Note that the last limit exists because in this case the sequence $(D_n)$ is nonincreasing as
\[D_{n+1} \le d(L_{n+1},M_n) + d(M_n,M_{n+1}) = D_n - D + d(M_n,M_{n+1}) \le D_n,\]
for all $n \in \N$.
\end{itemize}
Consequently, the man wins if and only if $D_n > D$ for all $n \in \N$ and $\lim_{n \to \infty}D_n > D$.

\begin{theorem}\label{thm-lion-dir-bd}
Let $(X,d)$ be a uniquely geodesic space and $A \subseteq X$ a nonempty and convex set where the Lion-Man game is played following the rules described above. If the lion always wins, then $A$ is directionally bounded.
\end{theorem}
\begin{proof}
Assume that there exists a directional curve $\gamma : [0, \infty) \to A$, i.e., there exists $b>0$ such that 
\[|s-t|-b\le d(\gamma(s),\gamma(t))\le |s-t|,\]
for all $s,t\geq 0$. Take $D=b$, $ L_0 =\gamma(0)$ and $M_n = \gamma((n+2)D + 1)$ for all $n \ge 0$. Then, for every $n \ge 0$,
\[d(M_n,M_{n+1}) \le (n+3)D + 1 - (n+2)D - 1 = D,\]
\[d(L_0,L_n) \le \sum_{i=0}^{n-1}d(L_i,L_{i+1}) \le nD\]
and
\[d(M_n,L_n) \ge d(M_n,L_0) - d(L_0,L_n) \ge (n+2)D + 1 - D - nD \ge D + 1,\]
hence the man wins.
\end{proof}

The subsequent result shows that if the Lion-Man game is played in a $\CAT(0)$ space, then the success of the man yields the existence of a local quasi-geodesic ray.

\setcounter{claim}{0}

\begin{proposition}\label{prop-man-ray}
Let $(X,d)$ be a $\CAT(0)$ space and $A \subseteq X$ a nonempty and convex set where the Lion-Man game is played. If the man wins, then for every $k > 0$ there exists a $k$-local $\sqrt{2}$-quasi-geodesic ray in $A$.
\end{proposition}
\begin{proof}
Suppose the man wins. For $n \in \N$ with $n \ge 1$, denote $\beta_{n} = \angle_{L_n}(L_{n-1},L_{n+1}) = \angle_{L_n}(L_{n-1},M_n)$. The following claim is also justified in \cite[p. 281]{AleBisGhr10}.
\begin{claim} \label{prop-man-ray-cl1}
$\displaystyle \lim_{n \to \infty} \beta_n = \pi$.
\end{claim}
\begin{proof}[Proof of Claim \ref{prop-man-ray-cl1}]
Let $\alpha_{n} =\angle_{L_n}(M_{n-1},M_n)$ for $n \ge 1$. Since the man wins, there exists $\alpha > 0$ such that $\lim_{n \to \infty} D_n = D + \alpha$. Moreover, $\lim_{n \to \infty}d(L_{n},M_{n-1}) = \alpha$. For the geodesic triangle $\Delta(L_n,M_{n-1},M_n)$, consider a comparison triangle $\overline{\Delta}(\overline{L}_n,\overline{M}_{n-1},\overline{M}_n)$ in $\R^2$ and denote its interior angle at $\overline{L}_n$ by $\overline{\alpha}_{n}$. As $d(M_{n-1},M_n) \le D$ for all $n \ge 1$, it follows that $\lim_{n \to \infty}\overline{\alpha}_{n} = 0$. Because $\alpha_n \le \overline{\alpha}_{n}$ for all $n \ge 1$, we have $\lim_{n \to \infty} \alpha_n = 0$. This implies $\lim_{n \to \infty} \beta_n = \pi$ as $\pi \le \beta_n + \alpha_n$.
\end{proof}

Let $k>0$. Then there exists $n_k \in \N$ such that for all $n\geq n_k$, 
\begin{equation}\label{prop-man-ray-eq1}
\beta_{n+1}\geq \pi - \frac{\pi}{4\left\lceil k/D \right\rceil }.
\end{equation}
Define $\gamma : [0,\infty) \to A$,
\[\gamma(t) = \left(1- \frac{t-nD}{D}\right)L_{n_k + n} + \frac{t-nD}{D}L_{n_k + n+1}, \quad \text{for } nD \le t < (n+1)D, \text{ where } n \in \N.\]
Note that $\gamma(nD) = L_{n_k +n}$ for all $n \in \N$. We show that $\gamma$ is a $k$-local $\sqrt{2}$-quasi-geodesic ray. Applying the triangle inequality,
\[d(\gamma(s),\gamma(t)) \le |s-t| \le \sqrt{2}|s-t|,\] 
for all $s,t \ge 0$. Thus, we only need to prove the following property.

\begin{claim} \label{prop-man-ray-cl2} For all $s,t \ge 0$ with $|s-t| \le k$,
\[d(\gamma(s),\gamma(t)) \ge \frac{\sqrt{2}}{2}|t-s|.\]
\end{claim}
\begin{proof}[Proof of Claim \ref{prop-man-ray-cl2}]
\end{proof}
First observe that for all $n\geq n_k$, by \eqref{prop-man-ray-eq1},
\[\left(\left\lceil\frac{k}{D}\right\rceil - 1\right)\pi + \frac{3\pi}{4} = \left\lceil \frac{k}{ D} \right\rceil\left(\pi - \frac{\pi}{4\left\lceil k/ D \right\rceil}\right) \le \sum_{i=1}^{\left\lceil k/D \right\rceil} \beta_{n+i} \le \left\lceil \frac{k}{ D} \right\rceil \pi.\]
For $n\geq n_k$ take 
\[B_n = \sum_{i=1}^{\left\lceil k/D \right\rceil} \beta_{n+i}- \left(\left\lceil \frac{k}{ D} \right\rceil -1\right)\pi.\] 
Then $ 3\pi/4 \le B_n \le \pi$ and so $\cos B_n \le -\sqrt{2}/2$.

Let $s,t \ge 0$ such that $0<  t-s \le k$. Fix $n\in \N$ such that $nD \le s < (n+1)D$ and $nD < t < \left(n+1+\left\lceil k/D \right\rceil\right) D$.  Denote $m=n_k +n$ and $B = B_m$. Then $\gamma(s) \in [L_m,L_{m+1}]$ with $s \ne L_{m+1}$,
\begin{equation}\label{prop-man-ray-eq2}
\cos B \le -\frac{\sqrt{2}}{2}.
\end{equation}
and 
\[B \leq \beta_{m+1} + \ldots + \beta_{m+i} - (i-1) \pi,\] 
for all $ i \in \left\{1,\ldots ,\left\lceil k/ D \right\rceil\right\}$.

We prove Claim \ref{prop-man-ray-cl2} by showing that $d(\gamma(s),\gamma(t)) \ge |\cos B| (t-s)$. Depending on the value of $t$ we distinguish several situations. For clarity, in each case we will index $t$.

Case I: $nD < t \le (n+1)D$. Denote $t_0 = t$.

Then $d(\gamma(s),\gamma(t_0))= t_0-s \ge  |\cos B|(t_0-s)$. 

Case II: $(n+1)D < t \le (n+2)D$. Denote $t_1 = t$.

Let $b_1 = d(\gamma(s),L_{m+1}) = (n+1)D - s$, $c_1 =  d(\gamma(t_1),L_{m+1}) = t_1 - (n+1)D$ and 
\[A_1 = \angle_{L_{m+1}}(\gamma(s),L_{m+2}) = \beta_{m+1} \ge B.\]
Consider a triangle $\Delta(x,y,z)$ in $\R^2$ so that $\|x-y\| = b_1$, $\|x-z\|=c_1$, and the interior angle at $x$ equals $A_1$. Since $X$ is a $\CAT(0)$ space, $d(\gamma(s),\gamma(t_1)) \ge \|y-z\|$. Applying the cosine law in $\R^2$ we get
\begin{align*} 
\|y-z\|^2 & = b_1^2 + c_1^2 - 2b_1c_1\cos A_1 \ge b_1^2 + c_1^2 - 2b_1c_1\cos B = b_1^2 + c_1^2 + 2b_1 c_1 |\cos B| \quad \text{by }\eqref{prop-man-ray-eq2}\\
& \geq \vert \cos B\vert (b_1 +c_1 )^2 = |\cos B|(t_1 -s)^2 \ge  |\cos B|^2(t_1 -s)^2.
\end{align*}
Hence, $d(\gamma(s),\gamma(t_1)) \ge  |\cos B|(t_1-s)$.

Case III: In general, assume $\left\lceil k/D \right\rceil >1$ and suppose that for $i\in \{1, \ldots, \left\lceil k/D \right\rceil - 1\}$,
\[A_i = \angle_ {L_{m+i}} (\gamma (s),L_{m+i+1})\geq  \beta_{m+1}+ \cdots +\beta_{m+i}- (i-1) \pi\]
and if $(n+i)D < t_i \le (n+i+1)D$,
\begin{equation}\label{prop-man-ray-eq3}
d(\gamma (s),\gamma (t_i ))\ge |\cos B|(t_i -s).
\end{equation}
We show that 
\[A_{i+1} = \angle_ {L_{m+i+1}} (\gamma (s),L_{m+i+2})\geq  \beta_{m+1}+ \cdots +\beta_{m+i+1}- i\pi\]
and if $(n+i+1)D < t_{i +1} \le (n+i+2)D$, 
\[d(\gamma (s),\gamma(t_{i +1}))\ge |\cos B|(t_{i +1}-s).\]

Because $\beta_{m+i+1} \le A_{i+1} + \angle_{L_{m+i+1}}(\gamma(s),L_{m+i}) \le A_{i+1} + \pi - A_i$ we have
\[A_{i+1} \ge \beta_{m+i+1} + A_i - \pi \ge \beta_{m+1}+ \cdots +\beta_{m+i+1}- i\pi \ge B.\]
Let $b_{i+1} = d(\gamma(s),L_{m+i+1}) \ge |\cos B|((n+i+1)D-s)$, where the last inequality follows by applying \eqref{prop-man-ray-eq3} with $t_i = (n+i+1)D$. Take also $c_{i+1} =  d(\gamma(t_{i+1}),L_{m+i+1}) = t_{i+1} - (n+i+1)D$. Consider a triangle $\Delta(x,y,z)$ in $\R^2$ so that $\|x-y\| = b_{i+1}$, $\|x-z\|=c_{i+1}$, and the interior angle at $x$ equals $A_{i+1}$. Then
\begin{align*} 
d(\gamma(s),\gamma(t_{i+1}))^2 &\ge \|y-z\|^2 = b_{i+1}^2 + c_{i+1}^2 - 2b_{i+1}c_{i+1}\cos A_{i+1}\ge b_{i+1}^2 + c_{i+1}^2 - 2b_{i+1}c_{i+1}\cos B\\
& = b_{i+1}^2 + c_{i+1}^2 + 2b_{i+1}c_{i+1}|\cos B| \ge b_{i+1}^2 + |\cos B|^2c_{i+1}^2 + 2b_{i+1}c_{i+1}|\cos B|\\
& = \left(b_{i+1} + |\cos B|c_{i+1}\right)^2\ge  |\cos B|^2(t_{i+1} -s)^2.
\end{align*}
Hence, $d(\gamma(s),\gamma(t_{i+1})) \ge  |\cos B|(t_{i+1}-s)$. This finishes the proof of the claim.
\end{proof}

We can now state our main result.
\begin{theorem}\label{main}
Let $A$ be a nonempty, closed and convex subset of a  complete $\CAT(0)$ space that is additionally $\delta$-hyperbolic. Then the following statements are equivalent:

\begin{itemize}
\item[(i)] $A$ is geodesically bounded;
\item[(ii)] $A$ is directionally bounded;
\item[(iii)] $A$ has the AFPP (for nonexpansive mappings);
\item[(iv)] $A$ has the FPP (for nonexpansive mappings);
\item[(v)] the lion always wins the Lion-Man game played in $A$. 
\end{itemize}
\end{theorem}
\begin{proof}
$(i) \Longleftrightarrow (ii)$: Follows from Proposition \ref{thm-geod-dir-bd}.

$(ii) \Longleftrightarrow (iii)$: Follows from \cite[Theorem 2.4]{Sha90}.

$(i) \Longleftrightarrow (iv)$: Follows from \cite[Theorem 3.1]{Pia18}.

$(v) \Longrightarrow (ii)$: Follows from Theorem \ref{thm-lion-dir-bd}.

$(i) \Longrightarrow (v)$: Suppose that $A$ is geodesically bounded and that the man wins. Let $k > 8\sqrt{2}M$, where $M=M(\delta,\sqrt{2})$ is given by Remark \ref{rmk-quasi-slim}. By Proposition \ref{prop-man-ray}, $A$ contains a $k$-local $\sqrt{2}$-quasi-geodesic ray. Using Corollary \ref{cor-local-quasi-ray} we obtain that $A$ contains a geodesic ray, a contradiction. 
\end{proof}

\subsection*{Final remarks and conclusions}
1. In \cite{LopNicPia18} we proved that in the setting of complete, locally compact, uniquely geodesic spaces, if the Lion-Man game is played in a closed and strongly convex domain $A$, then the lion always wins if and only if $A$ has the fixed point property for continuous mapping (i.e., every continuous self-mapping defined on $A$ has at least one fixed point). This equivalence is no longer true if the local compactness assumption is dropped because, on the one hand, if the Lion-Man game is played in a convex and bounded subset of a Hilbert space, then the lion always wins (see \cite{KohLopNic18}). On the other hand, the unit ball of a Hilbert space has the fixed point property for continuous mappings if and only if the space is finite dimensional.

Regarding the relation between the FPP (for nonexpansive mappings) and the solution of the Lion-Man game, it would be interesting to know if the $\delta$-hyperbolic condition could be removed from Theorem \ref{main}. A first approach to this problem might be to consider the Hilbert framework. Recall that, according to \cite{Ray80}, boundedness is a necessary and sufficient condition for a closed and convex subset of a Hilbert space to have the FPP. Thus  we can raise the following problem.
\begin{question}
If the lion always wins the Lion-Man game played in a closed and convex subset $A$ of a Hilbert space, must $A$ be bounded?
\end{question} 

\noindent 2. In this paper we considered an $\varepsilon$-capture criterion. Even in compact and convex subsets of the Euclidean plane, there exist games of this type where the lion wins by satisfying the condition $D_n > D$ for all $n \in \N$ and $\lim_{n\to \infty}D_n = D$ (see \cite[Example 1]{LopNicPia18}). Thus, for physical capture (i.e., the case when there exists $n_0 \in \N$ such that $D_{n_0} \le D$) we must assume some very rigid geometric conditions. This is the case when the domain of the game is a convex and geodesically bounded subset of an ${\R}$-tree. 

To see this, suppose that $D_n > D$ for all $n \in \N$. Note first that $d(L_n,L_{n+1}) = D$ for all $n \in \N$. We show that for every $n \in \N$, 
\begin{equation}\label{thm-R-tree-lion-eq1}
L_n \in [L_0, L_{n+1}].
\end{equation}
In this case, $L_i \in [L_0,L_{n}]$ for all $n \in \N$ and all $i \le n$. Therefore, $d(L_0, L_{n}) = nD$ for all $n \in \N$ and $\bigcup_{n \ge 0}[L_0,L_n]$ is a geodesic ray in $A$, which is a contradiction.

To prove that \eqref{thm-R-tree-lion-eq1} holds we use an inductive argument. For $n=0$ this is obvious. We suppose now that \eqref{thm-R-tree-lion-eq1} holds for $n=k$ and prove that it also holds for $n=k+1$. 

First note that as $L_k \in [L_0, L_{k+1}]$ and $L_{k+1} \in [L_k,M_k]$, we have $L_k, L_{k+1} \in [L_0,M_k]$. Take now $y \in A$ such that $[M_k,L_0] \cap [M_k,M_{k+1}] = [M_k,y]$. If $L_{k+1} \in [M_k,y]$, then $L_{k+1} \in [M_k,M_{k+1}]$ and so $d(L_{k+1},M_{k+1}) \le d(M_k,M_{k+1}) \le D$, which contradicts the assumption that $D_{k+1} > D$. Thus, $L_{k+1} \in [L_0,y]$.

It is easy to see that $y \in [L_0, M_{k+1}]$. Otherwise, if $[L_0,y] \cap [y,M_{k+1}] = [y,z]$ for zome $z \in A$, $z \ne y$, then $z \in [M_k,L_0] \cap [M_k,M_{k+1}] = [M_k,y]$. This is a contradiction because $z \in [y,M_{k+1}]$.

Because $L_{k+1} \in [L_0,y]$, we get $L_{k+1} \in [L_0, M_{k+1}]$. Recalling that $L_{k+2} \in [L_{k+1},M_{k+1}]$ we obtain \eqref{thm-R-tree-lion-eq1} for $n=k+1$.
\begin{question} 
What other geometric conditions imply physical capture?
\end{question}

\section{Acknowledgements}
This work was partially supported by DGES (Grant MTM2015-65242-C2-1P). \\


\begin{thebibliography}{99}

\bibitem{AleBisGhr06} 
S. Alexander, R. Bishop, R. Ghrist, Pursuit and evasion in non-convex domains of arbitrary dimension, in: Proc. of Robotics: Science $\&$ Systems (2006).

\bibitem{AleBisGhr10} 
S. Alexander, R. Bishop, R. Ghrist, Total curvature and simple pursuit on domains of curvature bounded above, Geom. Dedicata 149 (2010), 275--290.

\bibitem{ADL97} 
J.M. Ayerbe Toledano, T. Dom\'{\i}nguez Benavides, G. L\'opez Acedo, Measures of noncompactness in metric fixed point theory, Birkh\"{a}user, Berlin, 1997.

\bibitem{Bac12} 
M. Ba\v{c}\'{a}k, Note on a compactness characterization via a pursuit game, Geom. Dedicata 160 (2012), 195--197.

\bibitem{Bri99} 
M.R. Bridson, A. Haefliger, Metric spaces of non-positive curvature, Springer-Verlag, Berlin, 1999.

\bibitem{Bro65}
F.E. Browder, Fixed-point theorems for noncompact mappings in Hilbert space, Proc. Nat. Acad. Sci. U.S.A. 53 (1965), 1272--1276.

\bibitem{BuySch07}
S. Buyalo, V. Schroeder, Elements of asymptotic geometry, EMS Monographs in Mathematics, European Mathematical Society, Z\"{u}rich, 2007.

\bibitem{ChHI11} 
T.H. Chung, G.A. Hollinger, V. Isler, Search and pursuit-evasion in mobile robotics: A survey, Auton. Robots 31 (2011), 299--316.

\bibitem{Dom12}
T. Dom\'inguez Benavides, The failure of the fixed point property for unbounded sets in $c_0$, Proc. Amer. Math. Soc. 140 (2012), 645--650.

\bibitem{Db97}
T. Dobrowolski, W. Marciszewski, Rays and the fixed point property in noncompact spaces, Tsukuba J.  Math. 21 (1997), 97--112.

\bibitem{EspKir06}
R. Esp\'{\i}nola, W.A. Kirk, Fixed point theorems in $\mathbb{R}$-trees with applications to graph theory, Topology Appl. 153 (2006), 1046--1055.

\bibitem{EspPia13}
R. Esp\'inola, B. Pi\k{a}tek, The fixed point property and unbounded sets in $\CAT(0)$ spaces, J. Math. Anal. Appl. 408 (2013), 638--654.

\bibitem{GoeKir90} 
K. Goebel, W.A. Kirk, Topics in metric fixed point theory, Cambridge Studies in Advanced Mathematics, vol. 28, Cambridge University Press, Cambridge, 1990.

\bibitem{GoeRei84} 
K. Goebel, S. Reich, Uniform Convexity, Hyperbolic Geometry, and Nonexpansive Mappings, in: Pure and Applied Mathematics, Marcel Dekker, Inc., New York, Basel, 1984.

\bibitem{Goh65} D. G\" ohde, Zum Prinzip der kontraktiven Abbildung, Math. Nach., 30 (1965), 251-258.

\bibitem{Kir65} W.A. Kirk, A fixed point theorem for mappings which do not increase distances, Amer. Math. Monthly, 72 (1965), 1004-1006.

\bibitem{Kir03}
W.A. Kirk, Geodesic geometry and fixed point theory, in: D. Girela, G. L\'{o}pez, R. Villa
eds., Seminar of Mathematical Analysis, Proceedings, Universities of Malaga and Seville,
Sept. 2002-Feb. 2003, Universidad de Sevilla, Sevilla, 2003, pp. 195-225.

\bibitem{Kir04}
W.A. Kirk, Geodesic geometry and fixed point theory II, in: J. Garc\'{i}a-Falset, E. Llorens-
Fuster, B. Sims eds., Fixed Point Theory and its Applications, Yokohama Publ., Yokohama,
2004, pp. 113-142.

\bibitem{Ki04}
W.A. Kirk, Fixed point theorems in $\CAT(0)$ spaces and $\R$-trees, Fixed Point Theory Appl. 4 (2004), 231--244.

\bibitem{Klee55} 
V.L. Klee, Some topological properties of convex sets, Trans. Amer. Math. Soc. 78 (1955), 30--45.

\bibitem{KohLopNic18} 
U. Kohlenbach, G. L\'{o}pez-Acedo, A. Nicolae, A quantitative analysis of the ``Lion-Man'' game, arXiv:1806.04496 [math.MG].

\bibitem{Li86}
J.E. Littlewood, Littlewood's Miscellany (ed: B. Bollob\'{a}s), Cambridge University Press, Cambridge, 1986.

\bibitem{LopNicPia18} 
G. L\'{o}pez-Acedo, A. Nicolae, B. Pi\k{a}tek, ``Lion-Man'' and the  Fixed Point Property, Geom. Dedicata (in press, https://doi.org/10.1007/s10711-018-0403-9).

\bibitem{LopPia15} 
G. L\'{o}pez-Acedo, B. Pi\k{a}tek, Characterization of compact geodesic spaces, J. Math. Anal. Appl. 425 (2015), 748--757.

\bibitem{LopPia16} 
G. L\'{o}pez-Acedo, B. Pi\k{a}tek, Some remarks on a characterization of compactness by means of the fixed point property in geodesic spaces, J. Nonlinear Convex Anal. 7 (2016), 1259--1263.

\bibitem{Pap05} 
A. Papadopoulos, Metric Spaces, Convexity and Nonpositive Curvature, European Math. Soc., Z\" urich, 2005.

\bibitem{Pia15} 
B. Pi\k{a}tek, The fixed point property and unbounded sets in spaces of negative curvature, Israel J. Math. 209 (2015), 323--334.

\bibitem{Pia17}
B. Pi\k{a}tek, The behavior of fixed point free nonexpansive mappings unbounded sets in $\CAT(0)$ spaces, J. Math. Anal. Appl. 445 (2017), 1072--1083.     

\bibitem{Pia18} 
B. Pi\k{a}tek, On the fixed point property for nonexpansive mappings in hyperbolic geodesic spaces, J. Nonlinear Convex Anal. 19 (2018), 571--582.

\bibitem{Ray80}
W.O. Ray, The fixed point property and unbounded sets in Hilbert space, Trans. Amer. Math. Soc. 258 (1980), 531--537.
 
\bibitem{Rei76} 
S. Reich, The fixed point property for nonexpansive mappings, Amer. Math. Monthly 83 (1976), 266--268.

\bibitem{Rei80}
S. Reich, The fixed point property for nonexpansive mappings II, Amer. Math. Monthly 87 (1980), 292--294.

\bibitem{Rei83}
S. Reich, The almost fixed point property for nonexpansive mappings, Proc. Amer. Math. Soc. 88 (1983), 44--46.

\bibitem{Sha90} 
I. Shafrir, The approximate fixed point property in Banach and hyperbolic spaces, Israel J. Math. 71 (1990), 211--223.

\bibitem{Tak10}
W. Takahashi, J.-C. Yao, F. Kohsaka, The fixed point property and unbounded sets in Banach spaces, Taiwanese J. Math. 14 (2010), 733--742.

\bibitem{Yuf18}
O. Yufereva, Lion and Man Game in Compact Spaces, Dyn. Games Appl.  9 (2019), 281--292.

\end{thebibliography}
\end{document}